\documentclass[10pt, reqno]{amsart}
\RequirePackage{amssymb,amsmath,amsthm,color,   graphicx} 

\RequirePackage[pdftex, colorlinks, colorlinks=true, citecolor=black, menucolor=black, linkcolor=black]{hyperref}
\hypersetup{
	colorlinks=true, 
	linkcolor=black, 
	urlcolor=blue, 
	linktoc=all 
}
\theoremstyle{plain}
\newtheorem{theorem}{Theorem}[section]
\newtheorem{proposition}[theorem]{Proposition}
\newtheorem{definition}[theorem]{Definition}
\newtheorem{lemma}[theorem]{Lemma}

\newtheorem*{theorem*}{Theorem}
\newtheorem*{proposition*}{Proposition}
\newtheorem*{definition*}{Definition}
\newtheorem*{lemma*}{Lemma}
\newtheorem*{corollary*}{Corollary}
\newtheorem*{example*}{Example}
\newtheorem*{remark*}{Remark}

\DeclareSymbolFont{yhlargesymbols}{OMX}{yhex}{m}{n} \DeclareMathAccent{\yhwidehat}{\mathord}{yhlargesymbols}{"62}
\numberwithin{equation}{section}

\newcommand{\C}{\mathbb{C}}
\newcommand{\R}{\mathbb{R}}

\renewcommand{\phi}{\varphi}

\DeclareMathOperator{\sech}{sech}

\newcommand{\qtq}[1]{\quad\text{#1}\quad}

\newcommand{\pt}{\partial}

\usepackage{enumitem}
\usepackage{amsrefs}

\begin{document}

\title[Soliton transmission]{Transmission of fast solitons for the NLS with an external potential}

\author[C. Hogan]{Christopher C. Hogan}

\address{Department of Mathematics \& Statistics, Missouri University of Science \& Technology, Rolla, MO, USA}
\email{cch62d@mst.edu}

\author[J. Murphy]{Jason Murphy}

\address{Department of Mathematics \& Statistics, Missouri University of Science \& Technology, Rolla, MO, USA 
\newline
\text{ } 
\newline 
\text{ }\quad Department of Mathematics, University of Oregon, Eugene, OR, USA}
\email{jamu@uoregon.edu}

\begin{abstract} We consider the dynamics of a boosted soliton evolving under the cubic NLS with an external potential.  We show that for sufficiently large velocities, the soliton is effectively transmitted through the potential.  This result extends work of Holmer, Marzuola, and Zworski \cite{holmerFastSolitonScattering2007}, who considered the case of a delta potential with no bound states, and work of Datchev and Holmer \cite{datchevFastSolitonScattering2009}, who considered the case of a delta potential with a linear bound state.  
\end{abstract}

\maketitle

\section{Introduction}

We consider one-dimensional cubic nonlinear Schr\"odinger equations (NLS) with an external potential: 
\begin{equation}\label{NLSV}
i\partial_t u = -\tfrac12\partial_x^2 u + Vu - |u|^2 u.
\end{equation}
As is well-known, the underlying cubic NLS 
\begin{equation}\label{NLS}
i\partial_t u = -\tfrac12\partial_x^2 u - |u|^2 u
\end{equation}
admits a family of solitary wave solutions of the form
\begin{equation}\label{soliton}
u(t,x)=e^{ixv+it/2-itv^2/2}\sech(x-x_0-vt),\quad x_0,v\in\R.
\end{equation}
In this work, we show that under suitable assumptions on the external potential, a highly boosted soliton evolving under \eqref{NLSV} can be effectively transmitted through the potential. 

We were inspired to consider this problem by the work of Holmer, Marzuola, and Zworski \cite{holmerFastSolitonScattering2007}, who considered \eqref{NLSV} in the special case of a repulsive delta potential (i.e. with $V=q\delta_0$ for some $q>0$).  We will discuss this and other related works in more detail below after stating our main result.

We first describe the class of potentials that we consider in this work.  We refer the reader to Section~\ref{S:prelim} for further discussion.

\begin{definition}[Admissible]\label{D:admissible} We say $V:\R\to\R$ is admissible if the operator $H:=-\tfrac12\partial_x^2+V$ obeys the following:
\begin{itemize}
\item $H$ has no zero energy resonance;
\item $H$ has at most one simple negative eigenvalue $-\lambda$, with corresponding eigenfunction $\varphi$ satisfying $\varphi\in L^1\cap L^\infty$; and 
\item the potential obeys the pointwise bound 
\[
|V(x)|\lesssim \langle x\rangle^{-s}\qtq{for some} s>2.
\]
\end{itemize}
We call $s$ the decay parameter of $V$. 
\end{definition}

Our main result is the following: 

\begin{theorem}\label{T} Let $V:\R\to\R$ be admissible with decay parameter $s>2$ and fix $\delta\in(\tfrac12,\tfrac{s}{1+s})$.  The following holds for all $v$ sufficiently large:

Let $u$ denote the solution to \eqref{NLSV} with
\[
u(0,x) = e^{ivx}\sech(x-x_0)\qtq{for some} x_0\leq -v^{1-\delta},
\]
and define the corresponding solitary wave solution to \eqref{NLS} by 
\[
u_1(t,x) = e^{ixv+it/2-itv^2/2}\sech(x-x_0-vt).
\]

Then
\begin{equation}\label{Teq}
\|u-u_1\|_{L_t^\infty L_x^2([0,(1-\delta)\log v]\times\R)} \lesssim v^{-(2\delta-1)}. 
\end{equation}
\end{theorem}

Note that since $u_1$ is initially centered at $x_0\leq -v^{1-\delta}$ and travels at speed $v$ and the potential is concentrated near $x=0$, the soliton has effectively `moved past' the potential after times $t\geq \tfrac{|x_0|}{v}+1$; indeed, in this regime we have $x_0+vt\geq 1$.  We remark that the approximation in \eqref{Teq} holds well beyond this time.

Soliton scattering by external potentials is a problem of both mathematical and physical interest (see e.g. \cite{holmerFastSolitonScattering2007, datchevFastSolitonScattering2009, grimshawSolitonInterferometryVery2022, walesSplittingRecombinationBrightsolitarymatter2020, perelmanRemarkSolitonpotentialInteractions2009, floerNonspreadingWavePackets1986, bronskiSolitonDynamicsPotential2000, goodmanStrongNLSSolitondefect2004, CaoMalomed}).  The first mathematical works directly studying this phenomenon seem to be \cite{holmerFastSolitonScattering2007, datchevFastSolitonScattering2009}, which treat the case of the delta potential (in both the repulsive and attractive cases).  On the other hand, in the physical setting one often encounters more general external potentials. Thus it is of both mathematical and physical interest to extend the results of \cite{holmerFastSolitonScattering2007, datchevFastSolitonScattering2009} to more general potentials.  In this work, we are able to address algebraically decaying potentials (as in \cite{grimshawSolitonInterferometryVery2022}, for example), including potentials admitting a linear bound state.  Presently, our analysis is restricted to the high velocity regime in which the leading order behavior is simply transmission of the soliton.

More generally, one expects some component of the soliton to be transmitted through the potential and some component to be reflected by the potential.  Indeed, such a decomposition was obtained in \cite{holmerFastSolitonScattering2007, datchevFastSolitonScattering2009} for the case of the delta potential (i.e. $V=q\delta_0$ with $q\in\R\backslash\{0\}$). In particular, for $\tfrac{q}{v}$ fixed and $v$ sufficiently large, the work \cite{holmerFastSolitonScattering2007} shows that for times $t\in[\tfrac{|x_0|}{v}+1,c\log v]$ (i.e. after the soliton-potential interaction time), one can write
\begin{equation}\label{HMZ}
u(t) = u_T(t) + u_R(t) + \mathcal{O}_{L^2}(v^{-a})
\end{equation}
for some $a\in(0,\tfrac12)$. Here $u_T,u_R$ are the (appropriately translated and boosted) solutions to \eqref{NLS} with data given by $T(v)\sech$ and $R(v)\sech$, where $T$ and $R$ are the translation and reflection coefficients arising in the scattering theory for the underlying Schr\"odinger operator $-\tfrac12\partial_x^2+q\delta_0$.  Utilizing the inverse scattering theory for \eqref{NLS}, the authors of \cite{holmerFastSolitonScattering2007} further provide a precise description of $u_T$ and $u_R$ (up to an additional $L^\infty$ error term).  In \cite{datchevFastSolitonScattering2009}, the authors followed up on the work of \cite{holmerFastSolitonScattering2007} and established an analogous result in the setting of an attractive delta potential, in which case one must also contend with the presence of a linear bound state.

In our setting, we fix the external potential and, in order to close several perturbative arguments, we work in the large velocity regime  $v\gg 1$. In this setting, the transmission and reflection coefficients in the scattering theory for $-\tfrac12\partial_x^2+V$ obey the asymptotics
\[
T(v) = 1 + \mathcal{O}(v^{-1}) \qtq{and} R(v) = \mathcal{O}(v^{-1}) \qtq{as}v\to\infty
\]
(see e.g. \cite[Equation~(2.36)]{naumkinSharpAsymptoticBehavior2016}).  In particular, even if we track the reflected component of the solution, it is likely that this component is already smaller than the error term in \eqref{Teq}, which is of size $v^{-(2\delta-1)}$ for some $\delta<1$.  Similarly, replacing the transmitted component by the full soliton incurs an error smaller than the error term in \eqref{Teq}.  Thus we do not track these components individually; instead, we consider the full soliton as the leading order term. 

This observation ultimately simplifies the arguments needed to establish the result appearing in Theorem~\ref{T} (compared to the more refined decompositions in \cite{holmerFastSolitonScattering2007, datchevFastSolitonScattering2009}).  Indeed, the authors of \cite{holmerFastSolitonScattering2007} rely on an analysis of the linear propagation of boosted data (in the presence of the potential) to witness the appearance of the transmission and reflection coefficients during the `interaction phase' of the evolution (e.g. when $t\in[\tfrac{|x_0|}{v}-v^{-\delta},\tfrac{|x_0|}{v}+v^{-\delta}]$). In our setting, the entire interaction phase is treated perturbatively, essentially relying only on the fact that the duration of the interaction is short if the velocity is large.  Subsequently, we do not need to use inverse scattering to understand the NLS evolution of $\alpha\cdot\sech$ for $\alpha\in(0,1)$ as in \cite{holmerFastSolitonScattering2007}.  Instead, we need only understand the NLS evolution of $\sech$, which is simply the soliton \eqref{soliton}. 

In light of this discussion, one can see directly that Theorem~\ref{T} actually applies with any traveling wave solution to \eqref{NLS} (decaying as $|x|\to\infty$), rather than just the ground state soliton. For example, the result in Theorem~\ref{T} applies to excited solitons as well. We also assert that the analysis in this paper would apply equally well in the setting of finitely many linear bound states corresponding to negative eigenvalues (with appropriate modifications to the proofs of Theorem~\ref{T:GWP} and Proposition~\ref{P:iterate}).

We would also like to mention the related work of Perelman \cite{perelmanRemarkSolitonpotentialInteractions2009}, who considered one-dimensional NLS equations with exponentially decaying potentials admitting no bound states and generalized higher order nonlinearities of quintic type. In this work, the author established a decomposition into transmitted and reflected components, similar to the type established in \cite{holmerFastSolitonScattering2007}. In fact, the decomposition in \cite{perelmanRemarkSolitonpotentialInteractions2009} was proven to hold \emph{globally} in time, which represents a considerable improvement over the results of \cite{holmerFastSolitonScattering2007, datchevFastSolitonScattering2009} and the present paper.  We note, however, that the techniques in \cite{perelmanRemarkSolitonpotentialInteractions2009} rely in an essential way on the quintic nature of the nonlinearity. Indeed, for quintic and higher nonlinearities it is possible to utilize Strichartz estimates directly to run perturbative arguments globally in time.  The quintic nonlinearity also allows for the use of the pseudoconformal transformation, which plays a role in the analysis of \cite{perelmanRemarkSolitonpotentialInteractions2009}.  Establishing a global-in-time decomposition in the cubic setting remains an interesting but challenging problem. 

For some other works sharing similarities with the problem considered here, we refer the interested reader to \cite{floerNonspreadingWavePackets1986, bronskiSolitonDynamicsPotential2000, deiftLongtimeAsymptoticsSolutions2011, goodmanStrongNLSSolitondefect2004, holmerBreathingPatternsNonlinear2009, holmerSlowSolitonInteraction2007, holmerSolitonInteractionSlowly2008, holmerSolitonSplittingExternal2007, CaoMalomed}.  

Our proof of Theorem~\ref{T} follows the general scheme given in \cite{holmerFastSolitonScattering2007}.  In particular, our proof is based on iterating a perturbative result for the difference
\[
w(t):=u(t)-u_1(t),
\]
where $u_1$ is the NLS soliton.  This function satisfies an equation of the form
\begin{equation}\label{weqi}
i\partial_t w + \tfrac12\partial_x^2 w - Vw = -|w|^2 w + \mathcal{O}(w^2 u_1) + \mathcal{O}(wu_1^2) - Vu_1,
\end{equation}
with $w|_{t=0}=0$ (see Proposition~\ref{P:iterate}). Regarding \eqref{weqi} as a forced NLS with lower order perturbations, the proof of Theorem~\ref{T} relies essentially on proving $L_t^\infty L_x^2$ bounds for such equations in the small-data setting. 

As $u_1$ is smooth and decaying, the quadratic terms in \eqref{weqi} pose no essential difficulty.  The linear terms, on the other hand, cannot be  incorporated into the standard small-data continuity arguments unless one restricts to short time intervals.  In particular, we will ultimately obtain the estimate in \eqref{Teq} by iterating over small intervals.  The essential term in \eqref{weqi} is therefore the forcing term $Vu_1$, representing the interaction between the soliton and the potential.  Applying Strichartz estimates for the underlying Schr\"odinger equation with potential, one finds it suffices to obtain suitable estimates for $Vu_1$ in $L_t^1 L_x^2$. We refer the reader to Proposition~\ref{P:iterate} for more details. 

At this point, we follow the lead of \cite{holmerFastSolitonScattering2007} and consider three `phases': 
\begin{itemize}
\item[(i)] the pre-interaction phase $t\leq \tfrac{|x_0|}{v}-v^{-\delta}$, when the soliton has not yet reached the potential; 
\item[(ii)] the interaction phase $t\in[\tfrac{|x_0|}{v}-v^{-\delta},\tfrac{|x_0|}{v}+v^{-\delta}]$; and 
\item[(iii)] the post-interaction phase $t\geq \tfrac{|x_0|}{v}+v^{-\delta}$, when the soliton has moved beyond the potential.
\end{itemize}

In phases (i) and (iii), the quantity $Vu_1$ is small in $L_t^1 L_x^2$ due to the decay of the soliton and potential together with the fact that $|x_0+tv|>v^{1-\delta}$ during these phases. In phase (ii), the quantity is small due to the brevity of the time interval. 

In phase (i), the error is initially zero and the quantity $Vu_1$ obeys a bound of order $v^{-s(1-\delta)}$. At the end of this phase, we obtain a bound of $v^{-s(1-\delta)}$ for the difference $w$.  

In phase (ii), the quantity $Vu_1$ obeys a bound of order $v^{-\delta}$ (the length of the time interval). This bound again yields the size of the error at the end of this phase. 

In phase (iii), the quantity $Vu_1$ is initially of order $v^{-s(1-\delta)}$ and actually obeys improving bounds as $t$ grows and the soliton continues to move away from the potential. However, as we must propagate over short time intervals and our error estimate increases by a fixed constant with each iteration, we can only propagate $\sim\log v$ times before our errors grow out of control. In particular, we begin with an error of size $v^{-\delta}$ and propagate over an interval of length $(1-\delta)\log v$, leading to a final error estimate of $v^{-(2\delta-1)}$. 

The rest of this paper is organized as follows:  In Section~\ref{S:prelim}, we introduce notation and collect some preliminary results.  In Section~\ref{S:iteration}, we give the proof of the main result, Theorem~\ref{T}. 

\subsection*{Acknowledgements}

J.M. and C.H. were supported by NSF grant DMS-2137217.  We are grateful to Maciej Zworski for helpful discussions related to this problem.

\section{Preliminaries}\label{S:prelim}

We write $A\lesssim B$ to denote $A\leq CB$ for some $C>0$.  We write $L_t^q L_x^r$ for the standard mixed Lebesgue spaces, e.g.
\[
\|u\|_{L_t^q L_x^r(I\times\R)} = \bigl\|\,\|u(t)\|_{L_x^r(\R)}\,\|_{L_t^q(I)}. 
\]
We may omit the domain (i.e. $I\times\R$ or $\R$) when it is clear from context.  We use primes to denote H\"older duals, e.g. for $r\in[1,\infty]$ we let $r'\in[1,\infty]$ denote the solution to $\tfrac{1}{r}+\tfrac{1}{r'}=1$.

Let $H=-\tfrac12\partial_x^2 + V$.  A key ingredient in our analysis is the Strichartz estimate adapted to the linear Schr\"odinger equation $i\partial_t u = Hu$.  

To state the result precisely, we briefly introduce the notion of a zero energy resonance.  To this end, we recall the notion of \emph{Jost solutions}, which are the solutions $f_\pm(\lambda,\cdot)$ to 
\begin{equation}\label{ODE}
-\tfrac12\partial_x^2 f_\pm + Vf_\pm = \tfrac12\lambda^2 f_\pm
\end{equation}
obeying 
\[
\lim_{x\to\pm\infty} [f_\pm(\lambda,x)-e^{\pm i\lambda x}] = 0.
\]
If $V$ decays sufficiently rapidly, solutions exist for all $\lambda\in\R$. We further define the Wronskian of $f_\pm$ by
\[
W(\lambda) = f_+(\lambda,x)\partial_x f_-(\lambda,x) - f_-(\lambda,x)\partial_x f_+(\lambda,x), 
\]
which (by \eqref{ODE}) is in fact independent of $x$. While it is known that $W(\lambda)\neq 0$ for all $\lambda\neq 0$, whether or not $W(0)=0$ depends on the potential $V$.

\begin{definition}\label{def:resonance} We say $H$ has a zero energy resonance if and only if $W(0)=0$.
\end{definition}

Throughout this paper, we consider potentials that are admissible in the sense of Definition~\ref{D:admissible}.  In particular, we assume that $H$ has no zero energy resonance.  On the other hand, we do allow for the possibility of a linear bound state (i.e. a negative eigenvalue of $H$). Writing $\varphi$ for the $L^2$-normalized bound state, we denote the projections onto the discrete/continuous spectrum by
\[
P_d f = \langle f,\varphi\rangle \varphi,\qtq{and} P_c f = f-P_d f,
\]
respectively.  Here $\langle\cdot,\cdot\rangle$ denotes the standard $L^2$ inner product.

Recall from Definition~\ref{D:admissible} that

\begin{proposition}[Strichartz estimates, \cite{goldbergDispersiveEstimatesSchrodinger2004, wederLLp2000}]\label{Strichartz}

Suppose $V:\R\to\R$ satisfies $\langle x\rangle V\in L^1$ and $H:=-\tfrac12\partial_x^2 + V$ has no zero energy resonance.  Let $t_0\in\R$ and $I\ni t_0$ be a time interval.  

For any $(q,r)$, $(\tilde q,\tilde r)$ satisfying $4\leq q,\tilde q\leq\infty$ and $\tfrac{2}{q}+\tfrac{1}{r}=\tfrac{2}{\tilde q}+\tfrac{1}{\tilde r}=\tfrac{1}{2}$, we have
\begin{align*}
\|P_c e^{-itH}\varphi\|_{L_t^q L_x^r(I\times\R)}&\lesssim \|\varphi\|_{L^2}, \\
\biggl\| \int_{t_0}^t P_c e^{-i(t-s)H}F(s)\,ds\biggr\|_{L_t^q L_x^r(I\times\R)} & \lesssim \|F\|_{L_t^{\tilde q'}L_x^{\tilde r'}(I\times\R)}.
\end{align*}
\end{proposition}

We remark that the decay assumptions on $V$ in Definition~\ref{D:admissible} are sufficient to apply the estimates in Proposition~\ref{Strichartz}. 

Utilizing these Strichartz estimates, we can establish global well-posedness for \eqref{NLSV} with initial data in $L^2$:

\begin{theorem}[Global well-posedness]\label{T:GWP} Let $u_0\in L^2$.  Then there exists a unique, global solution $u:\R\times\R\to\C$ to \eqref{NLSV} with $u|_{t=0}=u_0$.  The solution satisfies the conservation of mass, i.e.
\[
M[u(t)]:=\int |u(t,x)|^2\,dx \equiv M[u_0]. 
\]
If $u_0\in H^1$, then the solution additionally satisfies the conservation of energy, i.e.
\[
E[u(t)]:=\int \tfrac14|\nabla u(t,x)|^2 + \tfrac12 |V(x)u(t,x)|^2 - \tfrac14|u(t,x)|^4\,dx \equiv E[u_0].
\]
\end{theorem}

Although Theorem~\ref{T:GWP} is a fairly standard subcritical well-posedness result (see e.g. the textbook of Cazenave \cite{cazenaveSemilinearSchrodingerEquations2003}), the presence of the linear bound state requires a slight modification of the usual proof.  Thus, we provide a sketch of local well-posedness here.

\begin{proof}[Sketch of the proof of Theorem~\ref{T:GWP}]  We prove local well-posedness up to some time $T=T(\|u_0\|_{L^2})$.  Using conservation of mass, we can then iterate to obtain the global-in-time result.

We denote $H = -\frac{1}{2}\partial_x^2 + V$ and $M = \|u_0\|_{L^2_x}$. Throughout, we will let $(-\lambda,\phi)$ denote the negative eigenvalue and linear bound state pair of $H$. We let $I = [0,T]$, where $T=T(M)$ will be determined below. We define the complete metric space $(B,d)$ by
\[
B = \{(f,a) : \|f\|_{L_t^\infty L_x^2\cap L_{t,x}^6} \leq 2CM,\quad \|a\|_{L^{\infty}_t} \leq 2CM\}
\]
where 
\[
d((f,a), (g,b)) = \|f - g\|_{L^{\infty}_tL^2_x} + \|a - b\|_{L^{\infty}_t},
\]
all space-time norms are over $I\times\R$, and $C\geq 1$ encodes implicit constants appearing in the estimates below (e.g. the implicit constants in Strichartz estimates and norms of the bound state). 

We study the local well-posedness of the coupled system
\begin{equation}\label{NLSVsys}
\begin{cases}
i\pt_t f = H f - P_c\bigl[|f + a \varphi|^2 (f + a \varphi)\bigr] \\
i\pt_t a = -\lambda a - \langle|f + a \varphi|^2 (f + a \varphi), \varphi \rangle,
\end{cases}
\end{equation}
with initial data $f|_{t=0} = f_0 = P_c u_0$ and $a|_{t=0} = a_0 = \langle u_0, \varphi \rangle$.  After constructing a solution $(f,\varphi)$ to \eqref{NLSVsys}, the solution to \eqref{NLSV} is then given by $u=f+a\varphi$, i.e. $f=P_c u$ and $a\varphi=P_d u$.

We define the map
\[
\Phi[f,a] = \bigl(\Phi_1[f,a], \Phi_2[f,a]\bigr),
\]
where 
\begin{align*}
\Phi_1[f,a] &= e^{-itH}f_0 + i \int_0^t e^{-i(t - s)H} P_c\bigl\{|f + a\varphi|^2(f + a\varphi)\bigr\} \, ds, \\
\Phi_2[f,a] &= e^{it\lambda}a_0 + i \int_0^t e^{i(t - s)\lambda} \langle|f + a\varphi|^2(f + a\varphi),\varphi \rangle \, ds. \\
\end{align*}
We prove that $\Phi:B\to B$ is a contraction; its unique fixed point is then the desired solution to \eqref{NLSVsys}.

To show $\Phi : B \rightarrow B$, observe that for $(f,a) \in B$, by Proposition~\ref{Strichartz} and H\"older's inequality,
\begin{equation*}
\begin{aligned}
\|\Phi_1 &[f,a]\|_{L_t^\infty L_x^2\cap L_{t,x}^6} \\
&\lesssim \|f_0\|_{L^2} + \| |f + a\varphi|^2(f + a\varphi) \|_{L^{\frac{6}{5}}_{t,x}(I \times \R)} \\
&\lesssim \|u_0\|_{L^2} + \| f^3 \|_{L^{\frac{6}{5}}_{t,x}(I \times \R)} + \| a^3 \varphi^3 \|_{L^{\frac{6}{5}}_{t,x}(I \times \R)} \\
&\lesssim \|u_0\|_{L^2_x} + |I|^{\frac{1}{2}}\| f \|^2_{L^{6}_{t,x}(I \times \R)} \| f \|_{L^{\infty}_t L^2_x(I \times \R)} + |I|^{\frac{5}{6}} \| a \|^3_{L^{\infty}_t(I)} \| \varphi \|^3_{L^{\frac{18}{5}}_x} \\
&\lesssim M + T^{\frac{1}{2}}M^3 + T^{\frac{5}{6}}M^3\lesssim 2M
\end{aligned}
\end{equation*}
for $T = T(M)$ small. Similarly, using H\"older's inequality, we obtain
\begin{equation*}
\begin{aligned}
\|\Phi_2 [f,a]\|_{L^{\infty}_t(I \times \R)} &\lesssim \|a_0\|_{L^{\infty}_t(I)} + \int_0^t | \langle |f + a\varphi|^2(f + a\varphi), \varphi \rangle | \, ds \\
&\lesssim \|a_0\|_{L^{\infty}_t(I)} + \int_0^t \| f^3 \varphi \|_{L^1_x} + \| a^3 \varphi^4 \|_{L^1_x} \, ds\\
&\lesssim \|u_0\|_{L^2_x} + |I|^{1/2}\| f \|^3_{L^{6}_{t,x}} \| \varphi \|_{L^2_x} + |I| \| a \|^3_{L^{\infty}_t} \| \varphi \|^4_{L^4_x} \\
&\lesssim M + T^{\frac{1}{2}}M^3 + TM^3 \lesssim 2M
\end{aligned}
\end{equation*}
for $T = T(M)$ small.

We next show that $\Phi$ is a contraction.  We let $(f,a),(g,b)\in B$.  We then have
\begin{align*}
d(\Phi[f,a],\Phi [g,b]) &\lesssim  \| |f + a\varphi|^2(f + a\varphi) - |g + b\varphi|^2(g + b\varphi) \|_{L^{\frac{6}{5}}_{t,x}} \\
& \quad + \| [|f + a\varphi|^2(f + a\varphi) - |g + b\varphi|^2(g + b\varphi)] \overline{\varphi} \|_{L^{1}_{t,x}} 
\end{align*}
One can now expand out the difference 
\[
|f+a\varphi|^2(f+a\varphi) - |g+b\varphi|^2(g+b\varphi),
\]
leading to a finite collection of terms that are of the form $vw(f-g)$ or $vw(a-b)\varphi$, where $v,w\in\{f,g,a\varphi,b\varphi\}$ (up to complex conjugation). Thus, using the same spaces as in the estimates above, we can derive an estimate of the form
\[
d(\Phi[f,a],\Phi[g,b])  \lesssim T^c M^2 d((f,a),(g,b))\leq \tfrac12 d((f,a),(g,b))
\]
for $T=T(M)$ sufficiently small.

\end{proof}

\section{Iteration}\label{S:iteration}

Throughout this section, we let $V:\R\to\R$ be admissible with decay parameter $s$ and let $\delta\in(\tfrac12,\tfrac{s}{1+s})$. We choose large $v\geq 1$ and let $u:\R\times\R\to\C$ denote the solution to \eqref{NLSV} with
\[
u(0,x) = e^{ivx}\sech(x-x_0)\qtq{for some} x_0\leq -v^{1-\delta}
\]
(cf. Theorem~\ref{T:GWP}). We also define the soliton solution to \eqref{NLS} by 
\begin{equation}\label{u1}
u_1(t,x) = e^{ixv+it/2-itv^2/2}\sech(x-x_0-vt).
\end{equation}

Our goal is to control the difference
\[
w(t):=u(t)-u_1(t),
\]
which satisfies the equation
\begin{equation}\label{weq}
\begin{aligned}
i\partial_t w + \tfrac12\partial_x^2 w - Vw & = -|w|^2w -2u_1|w|^2 - \overline{u_1}w^2 \\
& \quad\quad  - 2|u_1|^2 w - u_1^2\bar w - Vu_1
\end{aligned}
\end{equation}
with $w|_{t=0}=0$. 

We will prove our main theorem by iterating the following proposition, which allows us to propagate smallness of $w$. 

\begin{proposition}\label{P:iterate} There exist $c_0,c_1,c_2>0$ so that the following holds:  Let $I=[t_a,t_b]$ with $|I|\leq c_2$.  If
\[
\|w(t_a)\|_{L^2} + \|Vu_1\|_{L_t^1 L_x^2(I\times\R)} \leq c_0,
\]
then
\[
\|w\|_{L_t^\infty L_x^2(I\times\R)} \leq c_1 \bigl[\|w(t_a)\|_{L^2} + \|Vu_1\|_{L_t^1 L_x^2(I\times\R)}\bigr].
\]
\end{proposition}

\begin{proof} We let $c_0,c_1,c_2>0$ to be determined below. We will estimate $w$ in the space
\[
X=L^{\infty}_tL^2_x(I \times \R) \cap L_{t,x}^6(I \times \R).
\]

We decompose $w$ according to the spectrum of $H$, writing
\[
w(t) = P_d w(t) + P_c w(t),\quad P_d w(t) = \langle w(t),\varphi\rangle\varphi,
\]
where $\varphi$ is the linear bound state of $H$

We treat the component in the continuous spectrum by applying the Strichartz estimate (Proposition~\ref{Strichartz}).  Using the Duhamel formula for \eqref{weq}, Proposition~\ref{Strichartz}, H\"older's inequality, and and imposing $c_2\leq 1$, we obtain
\begin{equation}\label{Pc-part}
\begin{aligned}
\|P_c w\|_{X} & \lesssim \|w(t_a)\|_{L^2} + \|w^3\|_{L_{t,x}^{\frac65}} + \|u_1 w^2\|_{L_{t,x}^{\frac65}} + \|u_1^2 w\|_{L_{t,x}^{\frac65}} + \|Vu_1\|_{L_t^1 L_x^2} \\
& \lesssim \|w(t_a)\|_{L^2}+\|Vu_1\|_{L_t^1 L_x^2} + |I|^{\frac23}\|u_1\|_{L_t^\infty L_x^6}\|u_1\|_{L_t^\infty L_x^2}\|w\|_{L_{t,x}^6} \\
&\quad + |I|^{\frac12}\|w\|_{L_{t,x}^6}^2\bigl\{\|w\|_{L_t^\infty L_x^2} + \|u_1\|_{L_t^\infty L_x^2}\} \\
& \lesssim \|w(t_a)\|_{L^2}+\|Vu_1\|_{L_t^1 L_x^2}+c_2^{\frac23}\|w\|_{X}+\|w\|_X^2 + \|w\|_X^3. 
\end{aligned}
\end{equation}

We next consider the component in the discrete spectrum.  We observe that by the Duhamel formula for \eqref{weq}, we may write
\[
P_d w(t) = e^{i\lambda(t-t_a)}P_dw(t_a) -i\int_{t_a}^t e^{i\lambda(t-s)}P_d\bigl\{|w|^2 w + \mathcal{O}(w^2 u_1)+\mathcal{O}(w u_1^2)- Vu_1\bigr\}\,ds. 
\]
We next note that by the definition of $P_d$, H\"older's inequality, and the fact that $\varphi\in L^1\cap L^\infty$, we have
\[
P_d: L^p\to L^q \qtq{for all} 1\leq p,q\leq\infty. 
\]
Indeed,
\[
\|P_d f\|_{L^q} = |\langle f,\varphi\rangle|\, \|\varphi\|_{L^q} \lesssim \|f\|_{L^p} \|\varphi\|_{L^{q'}}\|\varphi\|_{L^q}\lesssim \|f\|_{L^p}. 
\]

It follows that $r\in \{2,6\}$ and $t\in I$, we have
\begin{align*}
\|P_d w(t)\|_{L^r} & \lesssim \|P_d w(t_a)\|_{L^r}+\int_{t_a}^t \|P_d\bigl\{|w|^2 w + \mathcal{O}(w^2 u_1)+\mathcal{O}(w u_1^2)- Vu_1\bigr\}\|_{L^r} \,ds \\
& \lesssim \|w(t_a)\|_{L^2} + \| |w|^2 w\|_{L_t^1 L_x^{\frac65}} + \| w^2 u_1\|_{L_t^1 L_x^{\frac65}} + \|w u_1^2 \|_{L_t^1 L_x^{\frac65}} + \|Vu_1\|_{L_t^1 L_x^2} \\
& \lesssim \|w(t_a)\|_{L^2} + \|Vu_1\|_{L_t^1 L_x^2} + |I|^{\frac56} \|u_1\|_{L_t^\infty L_x^6} \|u_1\|_{L_t^\infty L_x^2} \|w\|_{L_{t,x}^6} \\
& \quad + |I|^{\frac23}\|w\|_{L_{t,x}^6}^2\{\|w\|_{L_t^\infty L_x^2} + \|u_1\|_{L_t^\infty L_x^2}\} \\
& \lesssim \|w(t_a)\|_{L^2} + \|Vu_1\|_{L_t^1 L_x^2} + |c_2|^{\frac56} \|w\|_X + \|w\|_X^2 +\|w\|_X^3,
\end{align*}
so that
\[
\|P_d w\|_{L_t^\infty L_x^r} \lesssim \|w(t_a)\|_{L^2} + \|Vu_1\|_{L_t^1 L_x^2} + |c_2|^{\frac56} \|w\|_X + \|w\|_X^2 +\|w\|_X^3,\quad r\in\{2,6\}.
\]

Using the fact that $|I|\leq 1$, we now observe
\[
\|P_d w\|_{X} \lesssim \|P_d w\|_{L_t^\infty L_x^2} + |I|^{\frac16}\|P_d w\|_{L_t^\infty L_x^6} \lesssim \|P_d w\|_{L_t^\infty L_x^2}+\|P_d w\|_{L_t^\infty L_x^6}.
\]
Thus we obtain
\begin{equation}\label{Pd-part}
\|P_dw\|_X \lesssim \|w(t_a)\|_{L^2} + \|Vu_1\|_{L_t^1 L_x^2} + |c_2|^{\frac56} \|w\|_X + \|w\|_X^2 + \|w\|_X^3.
\end{equation}

Combining \eqref{Pc-part} and \eqref{Pd-part} and choosing $c_2$ sufficiently small, we derive that 
\[
\|w\|_X \leq C\{\|w(t_a)\|_{L^2}+\|Vu_1\|_{L_t^1 L_x^2}\} + C\{\|w\|_X^2+\|w\|_X^3\}
\]
for some $C>0$. Using a standard continuity argument, we find that if $c_0$ is sufficiently small, then
\[
\|w\|_X \leq 2C\{\|w(t_a)\|_{L^2}+\|Vu_1\|_{L_t^1 L_x^2}\},
\]
which implies the result with $c_1=2C$. \end{proof}

We will use the following straightforward estimate to control the soliton-potential interaction.

\begin{lemma}\label{L:error} For $s>\tfrac12$, we have
\[
\| e^{-|x-y|}\langle x\rangle^{-s}\|_{L^2} \lesssim \langle y\rangle^{-s}. 
\]
\end{lemma}

\begin{proof} We begin by estimating
\begin{align}
\| e^{-|x-y|}\langle x\rangle^{-s}\|_{L^2} & \lesssim \| e^{-|x-y|}\langle x\rangle^{-s}\|_{L^2(|x-y|\leq \frac12|y|)} \label{ex1}\\
& \quad + \| e^{-|x-y|}\langle x\rangle^{-s}\|_{L^2(|x-y|>\frac12|y|)}.\label{ex2} 
\end{align}
For the term \eqref{ex1}, we have $|x|\geq\tfrac12|y|$, so that
\[
\eqref{ex1} \lesssim \langle y\rangle^{-s}\|e^{-|x-y|}\|_{L^2} \lesssim \langle y\rangle^{-s}.
\]
For the term \eqref{ex2}, we have 
\[
\eqref{ex2} \lesssim e^{-\frac12|y|}\|\langle x\rangle^{-s}\|_{L^2}\lesssim e^{-\frac12|y|}\lesssim \langle y\rangle^{-s}.
\]
\end{proof}

We turn now to the proof of Theorem~\ref{T}.

\begin{proof}[Proof of Theorem~\ref{T}] We let $u,u_1,s$, and $\delta$ be as above and recall the constants $c_0,c_1$, and $c_2$ from Proposition~\ref{P:iterate}.  Following the framework of \cite{holmerFastSolitonScattering2007}, the proof proceeds in three phases.

\emph{Phase 1.}  Without loss of generality, we may choose an integer $N\geq 1$ such that
\[
Nc_2 = T_1 := \tfrac{|x_0|}{v}-v^{-\delta}.
\]
We will prove that
\begin{equation}\label{phase1}
\|u-u_1\|_{L_t^\infty L_x^2([0,T_1]\times\R)} \lesssim v^{-s(1-\delta)}. 
\end{equation}

We define the intervals 
\[
I_k=[kc_2,(k+1)c_2],\quad k=0,\dots,N-1.
\]
By \eqref{u1}, Definition~\ref{D:admissible}, and Lemma~\ref{L:error}, we have that 
\begin{align}
\|Vu_1\|_{L_t^1 L_x^2(I_k\times\R)} &\lesssim c_2 \| e^{-|x-[x_0+tv]|}\langle x\rangle^{-s}\|_{L_t^\infty L_x^2(I_k\times\R)} \nonumber \\
& \lesssim \|\langle x_0+tv\rangle^{-s}\|_{L_t^\infty(I_k)} \nonumber \\
& \leq C\langle x_0+(k+1)c_2v\rangle^{-s} \label{Vu1-phase1}
\end{align}
for all $k=0,\dots,N-1$. We now define the constants $\{A_k\}_{k=0}^{N-1}$ inductively via
\[
A_0 = Cc_1\langle x_0+c_2v\rangle^{-s},\quad A_{k+1}=c_1\bigl[A_k+C\langle x_0+(k+2)c_2 v\rangle^{-s}\bigr],
\]
so that 
\[
A_k = Cc_1^{k+2}\sum_{\ell=1}^{k+1}c_1^{-\ell}\langle x_0+\ell c_2 v\rangle^{-s}.
\]
We observe that by definition of $T_1$, we have
\[
\langle x_0+\ell c_2 v\rangle^{-s}\lesssim v^{-s(1-\delta)} \qtq{for all}\ell=1,\dots,N,
\]
which implies that
\begin{equation}\label{Akbd}
A_k\lesssim v^{-s(1-\delta)} \qtq{for all}k=1,\dots,N-1. 
\end{equation}
We will prove by induction that
\begin{equation}\label{induct1}
\|w\|_{L_t^\infty L_x^2(I_k\times \R)}\leq A_k
\end{equation}
for $k=0,\dots,N-1$, which (together with \eqref{Akbd}) yields \eqref{phase1}.  The plan is to iterate Proposition~\ref{P:iterate}. 

We turn to the induction argument. For the base case, we use the fact that $w(0)=0$ and 
\[
\|Vu_1\|_{L_t^1 L_x^2(I_0\times\R)} \leq C\langle x_0+c_2v\rangle^{-s} \lesssim v^{-s(1-\delta)}<c_0
\]
for $v$ sufficiently large. Thus, we may apply Proposition~\ref{P:iterate} on $I_0$ to obtain
\[
\|w\|_{L_t^\infty L_x^2(I_0\times\R)} \leq c_1\|Vu_1\|_{L_t^1 L_x^2(I_0\times\R)}\leq Cc_1\langle x_0+c_2 v\rangle^{-s} = A_0,
\]
which gives the base case. 

Now suppose that \eqref{induct1} holds up to level $k$ for some $k<N-1$.  In particular, recalling \eqref{Vu1-phase1} and \eqref{Akbd}, 
\[
\|w((k+1)c_2)\|_{L^2} + \|Vu_1\|_{L_t^1 L_x^2(I_{k+1}\times\R)}\leq A_k + C\langle x_0+(k+2)c_2v\rangle^{-s} \lesssim v^{-s(1-\delta)}<c_0
\]
for $v$ large.  Thus, we may apply Proposition~\ref{P:iterate} on $I_{k+1}$ to obtain
\[
\|w\|_{L_t^\infty L_x^2(I_{k+1}\times\R)} \leq c_1[A_k+C\langle x_0+(k+2)c_2v\rangle^{-s}] = A_{k+1},
\]
which completes the induction and hence the proof of \eqref{phase1}.

\emph{Phase 2.} We recall $T_1=\tfrac{|x_0|}{v}-v^{-\delta}$ and define
\[
T_2:=\tfrac{|x_0|}{v}+v^{-\delta}.
\]
We will prove that
\begin{equation}\label{phase2}
\|u-u_1\|_{L_t^\infty L_x^2([T_1,T_2]\times\R)}\lesssim v^{-\delta}.
\end{equation}
via an application of Proposition~\ref{P:iterate}.  To this end, first observe that 
\[
|[T_1,T_2]|=2v^{-\delta}\leq c_2
\]
for $v$ sufficiently large. Next, we have from \eqref{phase1} that
\[
\|w(T_1)\|_{L^2} \lesssim v^{-s(1-\delta)}<\tfrac12 c_0
\]
for $v$ large. Furthermore, by H\"older's inequality,
\begin{align*}
\|Vu_1\|_{L_t^1 L_x^2([T_1,T_2]\times\R)} & \lesssim v^{-\delta} \|V\|_{L_x^2}\|u_1\|_{L_{t,x}^\infty} <\tfrac12 c_0
\end{align*}
for $v$ large.  Thus we may apply Proposition~\ref{P:iterate} to obtain
\begin{align*}
\|w\|_{L_t^\infty L_x^2([T_1,T_2]\times\R)} & \lesssim \|w(T_1)\|_{L^2}+\|Vu_1\|_{L_t^1 L_x^2([T_1,T_2]\times\R)} \\
& \lesssim v^{-s(1-\delta)}+v^{-\delta} \lesssim v^{-\delta}
\end{align*}
(cf. $\delta<\tfrac{s}{1+s}$), which yields \eqref{phase2}. 

\emph{Phase 3.} We argue similarly to Phase 1. Without loss of generality, we may choose an integer $M\geq 1$ so that
\[
Mc_2 = (1-\delta)\log v
\]
and let 
\[
T_3 := T_2+(1-\delta)\log v =  \tfrac{|x_0|}{v}+v^{-\delta}+(1-\delta)\log v. 
\]
We will prove that
\begin{equation}\label{phase3}
\|u-u_1\|_{L_t^\infty L_x^2([T_2,T_3]\times\R)}\lesssim v^{-[2\delta-1]}.
\end{equation}

We define the intervals 
\[
J_k=[T_2+kc_2,T_2+(k+1)c_2],\quad k=0,\dots,M-1.
\]

By \eqref{u1}, Definition~\ref{D:admissible}, and Lemma~\ref{L:error}, we have
\begin{align}
\|Vu_1\|_{L_t^1 L_x^2(J_k\times\R)} & \lesssim c_2 \| e^{-|x-[x_0+tv]|}\langle x\rangle^{-s}\|_{L_t^\infty L_x^2(J_k\times \R)} \nonumber \\
& \lesssim \|\langle x_0+tv\rangle^{-s}\|_{L_t^\infty(J_k)} \\
& \leq C\langle x_0+(T_2+kc_2)v\rangle^{-s} \label{Vu1-phase3}
\end{align}
for all $k=0,\dots,M-1.$  We now define constants $\{B_k\}_{k=0}^{M-1}$ inductively via
\[
B_0 = Cc_1v^{-\delta},\quad B_{k+1}=c_1\bigl[[B_k+C\langle x_0+[T_2+(k+1)c_2]v\rangle^{-s}\bigr],
\]
so that
\[
B_k = Cc_1^{k+1}v^{-\delta} + Cc_1^{k+1}\sum_{\ell=1}^k c_1^{-\ell}\langle x_0+(T_2+\ell c_2)v\rangle^{-s}. 
\]

We observe that by definition of $T_2$, we have
\[
\langle x_0+(T_2+\ell c_2)\rangle^{-s} \lesssim v^{-s(1-\delta)} \qtq{for all}\ell=0,\dots,N-1.
\]
Thus, recalling that $c_2 M = (1-\delta)\log v$ and $\delta<\tfrac{s}{s+1}$, for $v$ sufficiently large we have
\begin{equation}\label{Bkbd}
B_k \lesssim v^{-[2\delta-1]}
\end{equation}

We will prove by induction that 
\begin{equation}\label{induct2}
\|w\|_{L_t^\infty L_x^2(J_k\times\R)}\leq B_k
\end{equation}
for $k=0,\dots,M-1$, which (together with \eqref{Bkbd}) yields \eqref{phase3}.  Once again, the plan is to iterate Proposition~\ref{P:iterate}.

We turn to the induction argument.  For the base case, enlarging the constant $C$ if necessary, we have from \eqref{phase2} that
\[
\|w(T_2)\|_{L^2} \leq \tfrac12 Cv^{-\delta}<\tfrac12 c_0
\]
for large $v$.  Furthermore, \eqref{Vu1-phase3} yields
\[
\|Vu_1\|_{L_t^1 L_x^2(J_0\times\R)}\lesssim \langle x_0+T_2v\rangle^{-s} \lesssim v^{-s(1-\delta)}<\tfrac12 c_0
\]
for $v$ large. Thus we may apply Proposition~\ref{P:iterate} on $J_0$ to obtain
\begin{align*}
\|w\|_{L_t^\infty L_x^2(J_0\times\R)} & \leq c_1[\|w(T_2)\|_{L^2}+\|Vu_1\|_{L_t^1 L_x^2(J_0\times\R)}] \\
&\leq c_1[\tfrac12 C v^{-\delta} + \tfrac12 C v^{-\delta}] \leq Cc_1 v^{-\delta}=B_0,
\end{align*}
where we have used $\delta<\tfrac{s}{s+1}$ and enlarged the constant $C$ if necessary. 

Now suppose that \eqref{induct2} holds up to level $k$ for some $k<M-1$.  In particular, recalling \eqref{Vu1-phase3} and \eqref{Bkbd}, we have
\begin{align*}
\|w((k+1)c_2)\|_{L^2}+\|Vu_1\|_{L_t^1 L_x^2(J_{k+1}\times\R)} &\leq B_k + C\langle x_0+[T_2+(k+1)c_2]v\rangle^{-s} \\
& \lesssim v^{-(2\delta-1)} + v^{-s(1-\delta)} < c_0
\end{align*}
for $v$ large.  Thus we may apply Proposition~\ref{P:iterate} on $J_{k+1}$ to obtain
\[
\|w\|_{L_t^\infty L_x^2(J_{k+1}\times\R)}\leq c_1[B_k + C\langle x_0+[T_2+(k+1)c_2]v\rangle^{-s}] = B_{k+1},
\]
which completes the induction and hence the proof of \eqref{phase3}. 

\emph{Conclusion.} Combining phases 1--3, we deduce that
\[
\|u-u_1\|_{L_t^\infty L_x^2([0,(1-\delta)\log v])} \lesssim v^{-s(1-\delta)}+ v^{-\delta} + v^{-(2\delta-1)} \lesssim v^{-(2\delta-1)},
\]
which completes the proof of Theorem~\ref{T}. 
\end{proof}


\bibliography{References}

\begin{bibdiv}
\begin{biblist}

\bib{bronskiSolitonDynamicsPotential2000}{article}{
      author={Bronski, J.~C.},
      author={Jerrard, R.~L.},
       title={Soliton dynamics in a potential},
        date={2000},
        ISSN={1073-2780},
     journal={Mathematical Research Letters},
      volume={7},
      number={2-3},
       pages={329\ndash 342},
      review={\MR{1764326}},
}

\bib{CaoMalomed}{article}{
      author={Cao, Xiang~D.},
      author={Malomed, Boris~A.},
       title={Soliton-defect collisions in the nonlinear {{Schr\"odinger}}
  equation},
        date={1995},
        ISSN={0375-9601},
     journal={Physics Letters A},
      volume={206},
      number={3},
       pages={177\ndash 182},
  url={https://www.sciencedirect.com/science/article/pii/0375960195006116},
}

\bib{cazenaveSemilinearSchrodingerEquations2003}{book}{
      author={Cazenave, Thierry},
       title={Semilinear {{Schr\"odinger}} equations},
      series={Courant {{Lecture Notes}} in {{Mathematics}}},
   publisher={{New York University, Courant Institute of Mathematical Sciences,
  New York; American Mathematical Society, Providence, RI}},
        date={2003},
      volume={10},
        ISBN={978-0-8218-3399-5},
      review={\MR{2002047}},
}

\bib{datchevFastSolitonScattering2009}{article}{
      author={Datchev, Kiril},
      author={Holmer, Justin},
       title={Fast soliton scattering by attractive delta impurities},
        date={2009},
        ISSN={0360-5302,1532-4133},
     journal={Communications in Partial Differential Equations},
      volume={34},
      number={7-9},
       pages={1074\ndash 1113},
      review={\MR{2560311}},
}

\bib{deiftLongtimeAsymptoticsSolutions2011}{article}{
      author={Deift, Percy},
      author={Park, Jungwoon},
       title={Long-time asymptotics for solutions of the {{NLS}} equation with
  a delta potential and even initial data},
        date={2011},
        ISSN={1073-7928,1687-0247},
     journal={International Mathematics Research Notices. IMRN},
      number={24},
       pages={5505\ndash 5624},
      review={\MR{2863375}},
}

\bib{floerNonspreadingWavePackets1986}{article}{
      author={Floer, Andreas},
      author={Weinstein, Alan},
       title={Nonspreading wave packets for the cubic {{Schr\"odinger}}
  equation with a bounded potential},
        date={1986},
        ISSN={0022-1236},
     journal={Journal of Functional Analysis},
      volume={69},
      number={3},
       pages={397\ndash 408},
      review={\MR{867665}},
}

\bib{goldbergDispersiveEstimatesSchrodinger2004}{article}{
      author={Goldberg, M.},
      author={Schlag, W.},
       title={Dispersive estimates for {{Schr\"odinger}} operators in
  dimensions one and three},
        date={2004},
        ISSN={0010-3616,1432-0916},
     journal={Communications in Mathematical Physics},
      volume={251},
      number={1},
       pages={157\ndash 178},
      review={\MR{2096737}},
}

\bib{goodmanStrongNLSSolitondefect2004}{article}{
      author={Goodman, Roy~H.},
      author={Holmes, Philip~J.},
      author={Weinstein, Michael~I.},
       title={Strong {{NLS}} soliton-defect interactions},
        date={2004},
        ISSN={0167-2789,1872-8022},
     journal={Physica D. Nonlinear Phenomena},
      volume={192},
      number={3-4},
       pages={215\ndash 248},
      review={\MR{2065079}},
}

\bib{grimshawSolitonInterferometryVery2022}{article}{
      author={Grimshaw, Callum~L.},
      author={Billam, Thomas~P.},
      author={Gardiner, Simon~A.},
       title={Soliton {{Interferometry}} with {{Very Narrow Barriers Obtained}}
  from {{Spatially Dependent Dressed States}}},
        date={2022-07},
        ISSN={0031-9007, 1079-7114},
     journal={Physical Review Letters},
      volume={129},
      number={4},
       pages={040401},
}

\bib{holmerSolitonSplittingExternal2007}{article}{
      author={Holmer, J.},
      author={Marzuola, J.},
      author={Zworski, M.},
       title={Soliton splitting by external delta potentials},
        date={2007},
        ISSN={0938-8974,1432-1467},
     journal={Journal of Nonlinear Science},
      volume={17},
      number={4},
       pages={349\ndash 367},
      review={\MR{2335125}},
}

\bib{holmerFastSolitonScattering2007}{article}{
      author={Holmer, Justin},
      author={Marzuola, Jeremy},
      author={Zworski, Maciej},
       title={Fast {{Soliton Scattering}} by {{Delta Impurities}}},
        date={2007-08},
        ISSN={1432-0916},
     journal={Communications in Mathematical Physics},
      volume={274},
      number={1},
       pages={187\ndash 216},
}

\bib{holmerSlowSolitonInteraction2007}{article}{
      author={Holmer, Justin},
      author={Zworski, Maciej},
       title={Slow soliton interaction with delta impurities},
        date={2007},
        ISSN={1930-5311,1930-532X},
     journal={Journal of Modern Dynamics},
      volume={1},
      number={4},
       pages={689\ndash 718},
      review={\MR{2342704}},
}

\bib{holmerSolitonInteractionSlowly2008}{article}{
      author={Holmer, Justin},
      author={Zworski, Maciej},
       title={Soliton interaction with slowly varying potentials},
        date={2008},
        ISSN={1073-7928,1687-0247},
     journal={International Mathematics Research Notices. IMRN},
      number={10},
       pages={Art. ID rnn026, 36},
      review={\MR{2429244}},
}

\bib{holmerBreathingPatternsNonlinear2009}{article}{
      author={Holmer, Justin},
      author={Zworski, Maciej},
       title={Breathing patterns in nonlinear relaxation},
        date={2009},
        ISSN={0951-7715,1361-6544},
     journal={Nonlinearity},
      volume={22},
      number={6},
       pages={1259\ndash 1301},
      review={\MR{2507321}},
}

\bib{naumkinSharpAsymptoticBehavior2016}{article}{
      author={Naumkin, I.~P.},
       title={Sharp asymptotic behavior of solutions for cubic nonlinear
  {Schr\"odinger} equations with a potential},
        date={2016-05},
        ISSN={0022-2488},
     journal={Journal of Mathematical Physics},
      volume={57},
      number={5},
       pages={051501},
}

\bib{perelmanRemarkSolitonpotentialInteractions2009}{article}{
      author={Perelman, Galina},
       title={A remark on soliton-potential interactions for nonlinear
  {{Schr\"odinger}} equations},
        date={2009},
        ISSN={1073-2780},
     journal={Mathematical Research Letters},
      volume={16},
      number={3},
       pages={477\ndash 486},
      review={\MR{2511627}},
}

\bib{walesSplittingRecombinationBrightsolitarymatter2020}{article}{
      author={Wales, Oliver~J.},
      author={Rakonjac, Ana},
      author={Billam, Thomas~P.},
      author={Helm, John~L.},
      author={Gardiner, Simon~A.},
      author={Cornish, Simon~L.},
       title={Splitting and recombination of bright-solitary-matter waves},
        date={2020-03},
        ISSN={2399-3650},
     journal={Communications Physics},
      volume={3},
      number={1},
       pages={51},
}

\bib{wederLLp2000}{article}{
      author={Weder, Ricardo},
       title={${L}^p$-${L}^{p'}$ estimates for the {Schr\"odinger} equation on
  the line and inverse scattering for the nonlinear {Schr\"odinger} equation
  with a potential},
        date={2000-01},
        ISSN={0022-1236},
     journal={Journal of Functional Analysis},
      volume={170},
      number={1},
       pages={37\ndash 68},
}

\end{biblist}
\end{bibdiv}

\end{document}